\documentclass[reqno]{amsart}
\usepackage{amssymb,amsmath,verbatim,amstext,eucal}
\usepackage[all]{xy}

\newtheorem{lemma}{Lemma}[section] 
 
\newtheorem{proposition}[lemma]{Proposition}
\newtheorem{theorem}[lemma]{Theorem}

\newtheorem{corollary}[lemma]{Corollary}

\newcommand{\FLEX}{\relax}
\newcommand{\flex}[1]{\renewcommand{\FLEX}{#1}}
\newtheorem{flexthm}[lemma]{\FLEX}

\newtheorem{prop}[lemma]{Proposition}

\newtheorem{cor}[lemma]{Corollary}

\theoremstyle{definition}
\newtheorem{definition}[lemma]{Definition}

\newenvironment{remark}[1]{\refstepcounter{lemma}%
\vskip 5pt \par\noindent {\bf #1\ \thelemma .}}{\vskip 5pt \par}

\newenvironment{remark*}[1]{\par \vskip 5pt \noindent 
{\bf #1.}}{\vskip 5pt \par}

\newlength{\dqlength}

\newcommand{\bh}{\ensuremath{{\mathcal B}({\mathcal H})}}

\newcommand{\cstar}{\hbox{$C^*$}}
\newcommand{\cstaralg}{$C^*$-algebra}

\newcommand{\dstext}[1]{\quad\text{#1}\quad}

\newcommand{\mean}{\mathop{
     \mathchoice{\vcenter{\hbox{\huge
           $\Lambda$}}}{\Lambda}{\Lambda}{\Lambda}}
     \displaylimits}

\newcommand{\norm}[1]{\left\|{#1}\right\|}

\providecommand{\qed}%
{\hfill \vrule height5pt width4pt depth1pt \vspace{+2.00ex}}

\newcommand{\ran}{\operatorname{range}}

\newcommand{\spn}{\operatorname{span}}
\newcommand{\supp}{\operatorname{supp}}

\newcommand{\bbC}{{\mathbb{C}}}

\newcommand{\bbN}{{\mathbb{N}}}

  \newcommand{\A}{{\mathcal{A}}}
  
  \newcommand{\B}{{\mathcal{B}}}

  \newcommand{\D}{{\mathcal{D}}}

  \newcommand{\G}{{\mathcal{G}}}
\renewcommand{\H}{{\mathcal{H}}}

  \newcommand{\M}{{\mathcal{M}}}
  \newcommand{\N}{{\mathcal{N}}}

\renewcommand{\S}{{\mathcal{S}}}
  
  \newcommand{\U}{{\mathcal{U}}}
  
  \newcommand{\W}{{\mathcal{W}}}

  \newcommand{\Z}{{\mathcal{Z}}}

\newcommand{\fH}{{\mathfrak{H}}}
\newcommand{\fI}{{\mathfrak{I}}}

\newcommand{\fP}{{\mathfrak{P}}}

\newcommand{\fS}{{\mathfrak{S}}}

 \newcommand{\pil}{\pi_\ell{}}
\newcommand{\pir}{\pi_r{}} \newcommand{\fn}{\mathfrak{n}}

\newcommand{\bimod}{\operatorname{bimod}}
\newcommand{\atom}{\operatorname{atom}}

\newcommand{\size}{\operatorname{card}}
\providecommand{\Zbl}[1]{#1.}

\begin{document}

\title[Lattices of Bures Bimodules]{Isomorphisms of Lattices of Bures-Closed Bimodules over Cartan
MASAs} \author[A. H. Fuller]{Adam H. Fuller} \address{Dept. of
Mathematics\\ University of Nebraska-Lincoln\\ Lincoln, NE\\
68588-0130 } \email{afuller7@math.unl.edu} \author[D.R. Pitts]{David
R. Pitts} \address{Dept. of Mathematics\\ University of
Nebraska-Lincoln\\ Lincoln, NE\\ 68588-0130}
\email{dpitts2@math.unl.edu}

\begin{abstract} For $i=1,2$, let $(\M_i,\D_i)$ be pairs consisting of
a Cartan MASA $\D_i$ in a von Neumann algebra $\M_i$, let
$\atom(\D_i)$ be the set of atoms of $\D_i$, and let $\fS_i$ be the
lattice of Bures-closed $\D_i$ bimodules in $\M_i$.  We show that when
$\M_i$ have separable preduals, there is a lattice isomorphism between
$\fS_1$ and $\fS_2$ if and only if the sets $\{(Q_1, Q_2)\in
\atom(\D_i)\times \atom(\D_i): Q_1\M_i Q_2\neq (0)\}$ have the same
cardinality.  In particular, when $\D_i$ is non-atomic, $\fS_i$ is
isomorphic to the lattice of projections in $L^\infty([0,1],m)$ where
$m$ is Lebesgue measure, regardless of the isomorphism classes of
$\M_1$ and $\M_2$.

\end{abstract}

\maketitle

\vskip 12pt

\section{Introduction} Let $\M$ be a von Neumann algebra containing a
Cartan MASA $\D$; we call $(\M,\D)$ a Cartan pair.  Feldman and
Moore~\cite{FeldmanMooreErEqReI, FeldmanMooreErEqReII} gave a
construction of Cartan pairs with separable preduals based on Borel
measurable equivalence relations and showed that all (separably
acting) Cartan pairs arise from their construction.  Building on the
work of Feldman and Moore \cite{FeldmanMooreErEqReI,
  FeldmanMooreErEqReII}, and Arveson \cite{ArvesonOpAlInSu}, Muhly,
Solel and Saito \cite{MuhlySaitoSolelCoTrOpAl} introduced the Spectral
Theorem for Bimodules. They claimed that if $\S$ is a $\sigma$-weakly
closed $\D$-bimodule of $\M$, then there is a Borel subset $\B$ of the
Feldman-Moore relation $R$ such that $\S$ consists of all those
operators in $\M$ whose ``matrices'' are supported in $\B$. That is, the
$\sigma$-weakly closed $\D$-bimodule $\S$ is determined precisely by
its support $\B$.

Unfortunately, there is a gap in the proof of the Spectral Theorem for
Bimodules; consult the ``added in proof'' portion of Aoi's paper
\cite{AoiCoEqSuInSu} for details.  While we are not aware of any complete
proof of the Spectral Theorem for Bimodules, Fulman
\cite{FulmanCrPrvNAlEqReThSu} has established it when $\M$ is
hyperfinite and $\M_*$ is separable.

In a recent paper, Cameron, the second author and Zarikian
\cite{CameronPittsZarikianCaMASAvNAlNoNePrMeTh}, introduced a new
perspective to the study of $\D$-bimodules in a Cartan pair $(\M,
\D)$. The approach in \cite{CameronPittsZarikianCaMASAvNAlNoNePrMeTh}
is operator theoretic and avoids the measure theoretic tools of
Feldman and Moore. In \cite{CameronPittsZarikianCaMASAvNAlNoNePrMeTh},
a version of the Spectral Theorem for
Bimodules is proved, not for $\sigma$-weakly closed bimodules but for
\emph{Bures}-closed $\D$-bimodules. In fact, it is shown that the
Spectral Theorem for Bimodules as introduced in
\cite{MuhlySaitoSolelCoTrOpAl} is true if and only if every
$\sigma$-weakly closed $\D$-bimodule is itself Bures-closed.

In this paper, we continue the study of the Bures-closed
$\D$-bimodules in a Cartan pair $(\M,\D)$. Our main result,
Theorem~\ref{samebbimod}, shows that the lattice of Bures-closed
bimodules for a separably acting Cartan pair $(\M,\D)$ depends upon:
i) whether $\D$ contains a diffuse part, and ii) the cardinality of
the restriction of 
the Murray-von Neumann equivalence relation for projections of $\M$ to
 the atoms of $\D$.  In this sense, the lattice of Bures-closed bimodules
depends suprisingly little on the Cartan pair $(\M,\D)$, when $\M$ is
separably acting. In particular,
if $(\M_1, \D_1)$ and $(\M_2, \D_2)$ are any two Cartan pairs in which
$\D_1$ and $\D_2$ are separably acting diffuse algebras, then they
share the same lattice structure of Bures-closed $\D$-bimodules.

Along the way, in Section~\ref{sec: proj}, we give a fuller
description of the supports of partial isometry normalizers of $\D$.
In particular, we describe a pre-order on $\G\N(\M,\D)$, the set 
consisting  of all 
partial isometry normalizers of $\D$, which is induced by their supports.

\section{Background and Preliminaries}\label{sec: prelim}

Let $\M$ be a von Neumann algebra. A MASA (maximal abelian
self-adjoint subalgebra) $\D$ in $\M$ is \emph{Cartan} if
\begin{enumerate}
 \item there is a faithful, normal conditional expectation
$E:\M\rightarrow\D$, and
 \item $\spn\{U\in\M:\ U\text{ is unitary and }U\D U^*=\D\}$ is
$\sigma$-weakly dense in $\M$.
\end{enumerate} If $\D$ is Cartan in $\M$ we call $(\M, \D)$ a
\emph{Cartan pair}. The set of \emph{normalizers} for $\D$ is the set
\begin{equation*} \N(\M, \D)=\{v\in\M: v^*\D v\cup v\D
v^*\subseteq\D\}.
\end{equation*} The \emph{groupoid normalizers}, denoted
$\G\N(\M,\D)$, are the elements of $\N(\M,\D)$ which are partial
isometries. Clearly, $\N(\M,\D)$ and $\G\N(\M,\D)$ are $\sigma$-weakly
dense in $\M$ when $(\M,\D)$ is a Cartan pair.

\begin{remark}{Notation} For any abelian von Neumann algebra $\W$,
$\atom(\W)$ will denote the set of atoms in $\W$.  Let $(\M,\D)$ be a
Cartan pair, and let $R_a$ be the restriction of the Murray-von
Neumann equivalence relation on projections of $\M$ to $\atom(\D)$.
For $A_1, A_2\in\atom(\D)$ write $A_1\sim A_2$ when $(A_1, A_2)\in
R_a$.
\end{remark}

Notice that if $v\in\M$ is a partial isometry such that $v^*v,
vv^*\in\atom(\D)$, then $v\in \G\N(\M,\D)$.  Indeed, for $d\in \D$,
$dv^*v\in\bbC v^*v$, so $vdv^*= vdv^*vv^*\in \bbC vv^*\subseteq \D$.
Likewise, $v^*dv\in\D$.

\begin{proposition}\label{summand} Let $(\M,\D)$ be a Cartan pair, and
set $X:=\sum_{Q\in\atom(\D)} Q$.  Then $X$ is a central projection of
$\M$.
\end{proposition}
\begin{proof} Let $U\in \G\N(\M,\D)$ be a unitary normalizer.  For
each $Q\in\atom(\D)$, $UQU^*\in\atom(\D)$, and the map $Q\mapsto
UQU^*$ is a permutation of $\atom(\D)$.  Hence $UXU^*=X$, and hence
$X$ commutes with $U$.  As $\M$ is generated by its unitary
normalizers, $X$ is in the center of $\M$.
\end{proof}

Thus, any Cartan pair decomposes as a direct sum of two Cartan pairs,
$(\M,\D)=(\M_c,\D_c)\oplus (\M_a,\D_a)$, where $(\M_c,\D_c)=(\M
X^\perp,\D X^\perp)$ and $(\M_a,\D_a)=(\M X,\D X)$.  Clearly,
$\atom(\D_c)=\emptyset$ and $\D_a$ is generated by its atoms.  We
shall call $(\M_c,\D_c)$ and $(\M_a,\D_a)$ the \textit{continuous} and
\textit{atomic} parts of $(\M,\D)$ respectively.

Henceforth, let $(\M,\D)$ be a Cartan pair with conditional
expectation $E$.  Fix a faithful normal semi-finite weight $\phi$ on
$\M$ such that $\phi\circ E=\phi$. We shall freely use notation from
\cite{TakesakiThOpAlII} (see pages 41-42 for  discussion  of
$\fn_\phi:=\{x\in\M: \phi(x^*x)<\infty\}$, 
Definition VII.1.5 for a discussion of the semi-cyclic
representation $(\pi_\phi,\fH_\phi,\eta_\phi)$, etc.).  The following
follows from the fact that $\phi\circ E=\phi$; details of the proof
are left for the reader.

\begin{lemma} \label{nphibi} With this notation, $\fn_\phi$ and
$\fn_\phi^*$ are $\D$-bimodules, and for $d\in \D$, $x\in
\fn_\phi$, and $y\in\fn_\phi^*$, we have
\begin{align*}\max\{\phi((dx)^*(dx)),\phi((xd)^*(xd))\}&\leq
\norm{d}^2\phi(x^*x)\dstext{and}\\
\max\{\phi((dy^*)^*(dy^*)),\phi((y^*d)^*(y^*d))\}&\leq
\norm{d}^2\phi(yy^*).
\end{align*}
\end{lemma}

\begin{definition}
Modifying~\cite[Definition~IX.1.13]{TakesakiThOpAlII} very slightly,
we will say that a quadruple $\{\pi,\fH, J,\fP\}$ is a
\textit{standard form for $\M$} if $\pi$ is a faithful normal
representation of $\M$ on $\fH_\pi$ and $\{\pi(\M),\fH,J,\fP\}$ is a
standard form for $\pi(\M)$ as in
\cite[Definition~IX.1.13]{TakesakiThOpAlII}.  Due to the uniqueness of
the standard form (see \cite[Theorem~IX.1.14]{TakesakiThOpAlII}), we
may, and sometimes will, assume without loss of generality that
$\{\pi(\M),\fH,J,\fP\}=\{\pi_\phi(\M),\fH_\phi, J_\phi, \fP_\phi\}$,
where $\phi$ is a faithful, semi-finite, normal weight on $\M$ such
that $\phi\circ E=\phi$.

When $(\M,\D)$ is a Cartan pair and $\{\pi,\fH,J,\fP\}$ is a standard
form for $\M$, define representations $\pil$ and $\pir$ of $\D$ on
$\fH_\pi$ by
\begin{equation}\label{pildef}
\pil(d)=\pi(d)\dstext{and}\pir(d)=J\pi(d^*)J,
\end{equation} and set
\begin{equation}\label{genZ} \Z=(\pil(\D)\cup\pir(\D))''.
\end{equation}
\end{definition}

The purpose of the following is to observe that $\Z$ is uniquely
determined.  The proof is an immediate consequence of
\cite[Theorem~1.4.7]{CameronPittsZarikianCaMASAvNAlNoNePrMeTh} 
and \cite[Theorem~IX.1.14]{TakesakiThOpAlII}.

\begin{proposition}\label{sameZ} Let $(\M,\D)$ be a Cartan pair.  For
$i=1,2$, suppose $\{\pi_i, \fH_i, J_i, \fP_i\}$ are standard forms for
$\M$, and let $\pir_i$ and $\Z_i$ be as in~\eqref{pildef}
and~\eqref{genZ}.  Then there exists a unique unitary operator
$U\in\B(\fH_1,\fH_2)$ such that
\begin{enumerate}
\item $\pi_2(x)=U\pi_1(x)U^*,\dstext{for all $x\in\M;$}$
\item $J_2=UJ_1U^*$;
\item $\fP_2=U\fP_1$;
\item $\pir_2(d)=U\pir_1(d)U^*,\quad d\in \D;$
\item $\Z_i$ is a MASA in $\B(\fH_i)$; and
\item $\Z_1\ni z\mapsto UzU^*$ is an isomorphism of $\Z_1$ onto
$\Z_2.$
\end{enumerate}
\end{proposition}

\subsection{Bimodules} The \emph{Bures topology}, see
\cite{BuresAbSuvNAl}, on $\M$ is the locally convex topology generated
by the family of seminorms
\begin{equation*} \{ T\mapsto\sqrt{\tau(E(T^*T))}:\ \tau\in(\D_*)^+\}.
\end{equation*} In this note we are primarily interested in the
Bures-closed $\D$-bimodules in $\M$. When the Cartan MASA is
understood, we will sometimes simply say ``bimodule'' in place of
``$\D$-bimodule.''  Any Bures-closed $\D$-bimodule is necessarily
$\sigma$-weakly closed.

It is shown in
\cite[Theorem~2.5.1]{CameronPittsZarikianCaMASAvNAlNoNePrMeTh} that if
$\S\subseteq \M$ is a non-zero Bures-closed $\D$-bimodule, then
$\S\cap\G\N(\M,\D)$ generates $\S$ as a Bures-closed bimodule.  We
will make frequent use of this fact.

Given a $\sigma$-weakly closed $\D$-bimodule $\S$ in $\M$, the
\emph{support of $\S$}, denoted $\supp(S)$, is the orthogonal
projection onto the $\Z$-invariant subspace,
$\overline{\pi_\phi(\S)\eta_\phi(\fn_\phi\cap\D)}$; as $\Z$ is a MASA
in $\B(\fH)$, $\supp(S)$ is a projection in $\Z$.  For an operator
$T\in\M$, define $\supp(T)$ to be the support of the Bures-closed
bimodule generated by $T$. The definition of the support of a bimodule
given here is as introduced in
\cite{CameronPittsZarikianCaMASAvNAlNoNePrMeTh}. The original concept
was introduced in \cite{MuhlySaitoSolelCoTrOpAl}.

For a partial isometry $w\in\G\N(\M, \D)$ we denote $\supp(w)$ by
$P_w$. Picking and choosing results from
\cite{CameronPittsZarikianCaMASAvNAlNoNePrMeTh} we have the following
alternative descriptions of $P_w$.

\begin{lemma}[{\cite[Lemma~1.4.6 and 
    Lemma~2.1.3]{CameronPittsZarikianCaMASAvNAlNoNePrMeTh}}] Given any
    $w\in\G\N(\M,\D)$ the following hold.
  \begin{enumerate}
  \item Let $\mean$ be an invariant mean on the (discrete) group of
unitaries in $\D$, $\U(\D)$ (which we may assume to satisfy
$\mean_{U\in \U(\D)} f(U) =\mean_{U\in\U(\D)} f(U^*)$ for every
$f\in\ell^\infty(\U(\D))$.  Then
    \begin{equation*} P_w=\mean_{U\in\U(\D)}\pil(wUw^*)\pir(U^*).
    \end{equation*}
  \item $P_w$ is the orthogonal projection onto
$\overline{\{\eta_\phi(wd):\ d\in\fn_\phi\cap\D\}}$, and for
$x\in\fn_\phi$,
    \begin{equation*} P_w\eta_\phi(x)=\eta_\phi(wE(w^*x)).
    \end{equation*}
  \end{enumerate}
\end{lemma}

We may view $\S\mapsto \supp(\S)$ as a map from the set of
$\D$-bimodules of $\M$ into the projection lattice of $\Z$.
Conversely, given a projection $Q$ in $\Z$, define a $\D$-bimodule,
$\bimod(Q)$, 
by
\begin{equation*} \bimod(Q)=\{T\in\M:\ \supp(T)\leq Q\}.
\end{equation*} It follows
from~\cite[Lemma~2.1.4(c)]{CameronPittsZarikianCaMASAvNAlNoNePrMeTh} that $\bimod(Q)$ is Bures-closed.
The operations $\bimod$ and $\supp$ satisfy the
following ``reflexivity-type'' condition.

\begin{theorem}[{\cite[Theorem~2.5.1]{CameronPittsZarikianCaMASAvNAlNoNePrMeTh}}]
  A $\D$-bimodule $\S\subseteq \M$ is 
 Bures-closed if and only if
  \begin{equation*} \S=\bimod(\supp(\S)).
  \end{equation*}
\end{theorem}

\section{Projections and Relations}\label{sec: proj} Throughout this
section, let $(\M,\D)$ be a Cartan pair with conditional expectation
$E$.  
Let $\{\pi,\fH, J,\fP\}$
be a standard form of $\M$, and construct the
maximal abelian algebra $\Z$ in $\B(\fH)$ as discussed in
section~\ref{sec: prelim}. 
We do not impose any condition of separable predual in this
section.

Our aim in this section is to better describe how the projections in
$\D$ relate to each other, in terms of the normalizers in
$\N(\M,\D)$. This in turn will provide us with a better description of
some of the projections in $\Z$. In Proposition~\ref{prop: relation},
we will describe exactly when $\pil(Q_1)\pir(Q_2)=0$ for projections
$Q_1, Q_2\in\D$. This will be determined by the existence of certain
normalizers in $\N(\M,\D)$. In the case of atomic projections in $\Z$
we will be able to go further. In Proposition~\ref{atomic} we will
show that the atomic projections of $\Z$ are completely determined by
the atomic projections in $\D$. This is a key tool in proving our main
result.

\begin{lemma}\label{lem1} Let $Q_1, Q_2$ be projections in $\D$. If
$w\in\G\N(\M, \D)$ then $ww^*\leq Q_1$ and $w^*w\leq Q_2$ if and only
if $w\in\bimod(\pil(Q_1)\pir(Q_2))$.
\end{lemma}

\begin{proof} Suppose $ww^*\leq Q_1$ and $w^*w\leq Q_2$. By
\cite[Lemma~2.1.4]{CameronPittsZarikianCaMASAvNAlNoNePrMeTh} it
suffices to show that
$\pi_\phi(w)\eta_\phi(\fn_\phi\cap\D)\subseteq\ran(\pil(Q_1)\pir(Q_2))$. Take
any $d\in\fn_\phi\cap\D$. Then
  \begin{align*}
\pil(Q_1)\pir(Q_2)\pi_\phi(w)\eta_\phi(d)&=\eta_\phi(Q_1wdQ_2)\\
&=\eta_\phi(Q_1wQ_2d)\\ &=\eta_\phi(wd)\\ &=\pi_\phi(w)\eta_\phi(d).
  \end{align*} Hence
$\pi_\phi(w)\eta_\phi(\fn_\phi\cap\D)\subseteq\ran(\pil(Q_1)\pir(Q_2))$
and $w\in\bimod(\pil(Q_1)\pir(Q_2))$.

  Conversely, suppose $w\in\G\N(\M,
\D)\cap\bimod(\pil(Q_1)\pir(Q_2))$. Let $v= w-Q_1 w Q_2$. We will be
done once we show that $v=0$. Note that, again by \cite[Lemma
2.1.4]{CameronPittsZarikianCaMASAvNAlNoNePrMeTh},
$\pi_\phi(w)\eta_\phi(d)\in\ran(\pil(Q_1)\pir(Q_2))$. Thus for
$d\in\fn_\phi\cap\D$ we have
  \begin{align*}
\pi_\phi(v)\eta_\phi(d)&=\pi_\phi(w)\eta_\phi(d)-\pi_\phi(Q_1wQ_2)\eta_\phi(d)\\
&=\pi_\phi(w)\eta_\phi(d)-\eta_\phi(Q_1wdQ_2)\\
&=\pi_\phi(w)\eta_\phi(d)-\pil(Q_1)\pir(Q_2)\pi_\phi(w)\eta_\phi(d)\\
&=0
  \end{align*} Hence $\eta_\phi(vd)=0$ for all
$d\in\fn_\phi\cap\D$. By the faithfulness of $\phi$ it follows that
$vd=0$ for all $d\in\fn_\phi\cap\D$. Since $\fn_\phi\cap\D$ is
weak-$*$ dense in $\D$, it follows that $vd=0$ for every
$d\in\D$. Hence $v=0$.
\end{proof}

We will give a more complete description of the relationship between
the projections $\{P_v: v\in\G\N(\M,\D)$ and the partial isometries in
$\G\N(\M,\D)$ in Lemma~\ref{lem2}. For the time being, Lemma~\ref{lem1}
gives the following statement.

\begin{cor}\label{cor1} If $u,v\in\G\N(\M,\D)$ and $P_u=P_v$, then
$vv^*=uu^*$ and $v^*v=u^*u$.
\end{cor}

\begin{proof} This follows immediately from Lemma~\ref{lem1}:  if
$P_u\leq P_v$, then, since $P_v\leq \pil(vv^*)\pir(v^*v)$,  
$uu^*\leq vv^*$ and $u^*u\leq v^*v$.
\end{proof}

\begin{prop}\label{prop: relation} For any two projections $Q_1$ and
$Q_2$ in $\D$, the following are equivalent:
  \begin{enumerate}
   \item $\pil(Q_1)\pir(Q_2)\neq0$;
   \item there is a non-zero $v\in\G\N(\M, \D)$ such that $vv^*\leq
Q_1$ and $v^*v\leq Q_2$;
   \item $Q_1\M Q_2\neq\{0\}$;
   \item there exists a $\sigma$-weakly closed $\D$-bimodule
     $\S\subseteq \M$ such that $Q_1\S Q_2\neq \{0\}$.
  \end{enumerate}
\end{prop}

\begin{proof} The equivalence of (a) and (b) follows immediately from
Lemma~\ref{lem1}, since $\pil(Q_1)\pir(Q_2)\neq0$ implies that
$\bimod(\pil(Q_1)\pir(Q_2))\neq\{0\}$.

Suppose (b) holds.  Given $v\in\G\N(\M, \D)$ such that $vv^*\leq Q_1$ and
$v^*v\leq Q_2$, we have $Q_1vQ_2=v$ and so $Q_1\M
Q_2\neq\{0\}$.  This gives  (c). 

That (c) implies (d) is obvious.  Finally suppose (d) holds.
Clearly, $Q_1\S Q_2$ is a  $\sigma$-weakly closed $\D$-bimodule.
By~\cite[Proposition~1.3.4]{CameronPittsZarikianCaMASAvNAlNoNePrMeTh},
there exists $0\neq v\in\G\N(\M,\D)\cap Q_1\S Q_2.$  Then $v=Q_1
vQ_2$,   $vv^*\leq Q_1$ and $v^*v\leq Q_2$.  Hence (b) holds, and we
are done.
\end{proof}

The support projections $\{P_w\in\Z:\ w\in\G\N(\M,\D)\}$ have a
natural partial ordering on them, induced from the partial ordering of
the projections in $\Z$.  This ordering imposes a pre-order on
$\G\N(\M,\D)$,  which  we now describe.

\begin{definition} For $u, v\in\G\N(\M,\D)$ we write $u\leq_\D v$ if
there is a $d$ in $\D$ such that $u=vd$.
\end{definition}

It is not hard to see that $\leq_\D$ is a pre-order on
$\G\N(\M,\D)$. For any scalar $\lambda$ with $|\lambda|=1$ and
$v\in\G\N(\M,\D)$ we have $v\leq_\D\lambda v$ and $\lambda v\leq_\D v$ so
$\leq_\D$ is indeed a pre-order and not a partial order. Recall that
partial isometries come already equipped with a partial ordering
$\leq$ (see, for example, \cite{HalmosHiSpPrBk}). In this order $u\leq
v$ if and only if there is a projection $P$ such that $u=vP$. In fact,
$P$ can be chosen to be $u^*u$. It follows that if $u,
v\in\G\N(\M,\D)$ and $u\leq v$ then $u\leq_\D v$. Hence the ordering
$\leq_\D$ is coarser than the usual ordering $\leq$ on partial
isometries.

\begin{lemma}\label{Dorder}
 Let $u,v\in\G\N(\M,\D)$. 
 If  $u\leq_\D v$,  then $v^*u\in\D$ and 
$u=v(v^*u)$.

\end{lemma}

\begin{proof} As $u\leq_\D v$ there is a $d\in\D$ such that
$u=vd$. Then $v^*u=v^*vd$ is in $\D$. Let
  \begin{equation*} a=vd-vv^*u.
  \end{equation*} We will show $a=0$. We have,
  \begin{align*} a^*a&=(vd-vv^*u)^*(vd-vv^*u)\\
&=d^*v^*vd-d^*v^*u-u^*vd+u^*vv^*u\\
&=d^*v^*vd-d^*v^*vd-d^*v^*vd+d^*v^*vv^*vd & \text{(since $u=vd$)}\\ &=0.&\qedhere
  \end{align*} 
\end{proof}

Now we relate the pre-ordering to  supports of elements of $\G\N(\M,\D)$.
\begin{lemma}\label{lem2} Let $u,v\in\G\N(\M,\D)$. The following 
 are equivalent:
\begin{enumerate}
\item $u\leq_\D v$;
\item  $u=vE(v^*u)$; and 
\item $P_u\leq P_v$.
\end{enumerate}
\end{lemma}

\begin{proof} The previous lemma shows that (a) implies (b).

Suppose next that 
$u=vE(v^*u)$. For any $x\in\fn_\phi$ we see that
  \begin{align*} P_vP_u\eta_\phi(x)&=P_v\eta_\phi(uE(u^*x))\\
&=\eta_\phi(vE(v^*u)E(u^*x))\\ &=\eta_\phi(uE(u^*x))=P_u\eta_\phi(x).
  \end{align*} Thus $P_u\leq P_v$.  So (b) implies (c).

  Now assume that $P_u\leq P_v$ for some $u, v\in\G\N(\M,\D)$. We aim
to show that $u=vE(v^*u)$. As $P_u\leq P_v$ a similar calculation to
above shows that, for all $x\in\fn_\phi$
  \begin{equation*} \eta_\phi(vE(v^*u)E(u^*x))=\eta_\phi(uE(u^*x)).
  \end{equation*} As $\phi$ is faithful, it follows that for all
$x\in\fn_\phi$ we have
  \begin{equation*} vE(v^*u)E(u^*x)=uE(u^*x).
  \end{equation*} As $\fn_\phi$ is weak-$*$ dense in $\M$ and $E$ is
normal, the above equation holds for all $x\in\M$. In particular it
holds when $x=u$, and hence $u=vE(v^*u)\in v\D$.  Therefore $u\leq_\D v$.
\end{proof}

We now classify the atomic projections of $\Z$ in terms of the atomic
projections in $\D$.

\begin{prop}\label{atomic} Let $(\M,\D)$ be a Cartan pair with
standard form $\{\pi,\fH,J,\fP\}$.  The following statements hold.
   \begin{enumerate}
\item If $A\in\atom(\Z)$, then there exist unique $Q_1, Q_2\in
\atom(\D)$ such that $A=\pil(Q_1)\pir(Q_2)$.  In addition, $Q_1\sim
Q_2$.
\item For $i=1,2,$ suppose $Q_i\in\atom(\D)$ and $Q_1\sim Q_2$.  Then
$\pil(Q_1)\pir(Q_2)\in \atom(\Z).$
\item The algebras $\Z$ and $\D$ satisfy,
$$\sum_{Q\in\atom(\D)}Q < I_\D \dstext{if and only if}
\sum_{A\in\atom(\Z)}A<I_\Z.$$
\end{enumerate}
\end{prop}

\begin{proof}\textit{Part a)} Take any non-zero $A\in\atom(\Z)$. As
$A\neq 0$ there is a non-zero $v\in\G\N(\M,\D)\cap\bimod(A)$. Hence
$P_v$ is a non-zero projection satisfying $P_v\leq A$. As $A$ is
atomic it follows that $P_v=A$. Let $Q_1=vv^*$ and $Q_2=v^*v$.
Obviously, $Q_1\sim Q_2$.  We aim
to show that $Q_1$ and $Q_2$ are atoms of $\D$.

Suppose that $P\leq Q_1$ is a non-zero projection in $\D$. Let $u=Pv$.
Then $uu^*=P$ and $u^*u\leq Q_2$. We also have that $u\in\G\N(\M,\D)$
and $u\leq v$. Hence $u\leq_\D v$. By Lemma~\ref{lem2}, we have that
$P_u\leq P_v$. As $u$ is non-zero and $P_v$ is atomic it follows that
$P_u=P_v$. By Corollary~\ref{cor1}, $P=Q_1$, and thus
$Q_1\in\atom(\D)$.  A similar argument shows $Q_2\in\atom(\D)$.

Recall that a projection $B$ in an abelian von Neumann algebra $\W$
belongs to $\atom(\W)$ if and only if $\W B$ is one-dimensional. Thus,
for any $h,k\in\D$, we have
\[\pil(h)\pir(k)\pil(Q_1)\pir(Q_2)\in\bbC \pil(Q_1)\pir(Q_2).\] Since
$\Z=\overline{\spn}^{\text{weak-}*}\{\pil(h)\pir(k): h,k\in\D\}$, it
follows that $\pil(Q_1)\pir(Q_2) \Z$ is one-dimensional and hence
$\pil(Q_1)\pir(Q_2)\in\atom(\Z)$. As $A=P_v\leq \pil(Q_1)\pir(Q_2)$, it
follows that $A=\pil(Q_1)\pir(Q_2)$. Uniqueness of $Q_1$ and $Q_2$
follows from Corollary~\ref{cor1}.

\textit{Part b)} Suppose that for $i=1,2,$ $Q_i\in\atom(\D)$ and
$Q_1\sim Q_2$.  Let $v\in\M$ be a partial isometry so that $vv^*=Q_1$
and $v^*v=Q_2$.  As observed earlier, $v\in\G\N(\M,\D)$.  Hence,
$\pil(Q_1)\pir(Q_2)\neq 0$ by Proposition~\ref{prop: relation}.   An
argument similar to that used in part (a) shows $\pil(Q_1)\pir(Q_2)\in\atom(\Z)$.

\textit{Part c)} Parts (a) and (b) show that $\atom(\D)$ and
$\atom(\Z)$ are both empty or both non-empty.  If both are empty, part
(c) holds trivially.

Assume then that $\atom(\Z)$ and $\atom(\D)$ are both non-empty. Let
$$X:= \sum_{Q\in \atom(\D)}Q\dstext{and} Y:=\sum_{A\in\atom(\Z)}A.$$

Suppose $X<I_\D$.  Then $\pil(I_\D-X)\neq 0$, and if $A\in\atom(\Z)$,
$A\pil(I_\D-X)=0$ by part (a).  So $\pil(I_\D-X) < I_\Z-Y$; hence $Y
<I_\Z$.

Conversely, suppose $Y<I_\Z$.  Then $0\neq\bimod(I_\Z-Y)$, so there
exists $0\neq v\in\G\N(\M,\D)\cap \bimod(I_\Z-Y)$. Hence $P_v\leq(I_\Z-Y)$.

Suppose there is a $Q\in\atom(\D)$ such that $vQ\neq 0$. Let $w=vQ\in\G\N(\M,\D)$. Clearly $w\leq_\D v$, and so by Lemma \ref{lem2} we have $P_w\leq P_v$. Note that $w^*w=Q$ is in $\atom(\D)$ and hence $ww^*$ is in $\atom(\D)$. By part (b) and Proposition~\ref{prop: relation} we have that $\pil(ww^*)\pir(w^*w)$ is a non-zero projection in $\atom(\Z)$. However, as $P_w\leq\pil(ww^*)\pir(w^*w)$ and $\pil(ww^*)\pir(w^*w)$ is atomic, it follows that $P_w=\pil(ww^*)\pir(w^*w)$. Hence we have
\begin{equation*}
 \pil(ww^*)\pir(w^*w)=P_w\leq P_v\leq(I_\Z-Y).
\end{equation*}
This is a contradiction. Hence $vQ=0$ for every $Q\in\atom(\D)$. As $v$ is non-zero it follows that $X< I_\D$.
\end{proof}

The following description of $R_a$ is an immediate consequence of
Propositions~\ref{prop: relation} and~\ref{atomic}.
\begin{corollary}\label{eqrelchar}
\[R_a=\{(Q_1,Q_2)\in\atom(\D)\times\atom(\D): Q_1\M Q_2\neq (0)\}.\]
\end{corollary}

\section{Main Result}\label{sec: main result} In this section, we
prove our main result, Theorem~\ref{samebbimod}.  This result shows
when the Cartan pair $(\M,\D)$ has separable predual, the isomorphism
class of the family of Bures-closed bimodules for $(\M,\D)$ depends
mostly upon the atomic part $(\M_a,\D_a)$ of $(\M,\D)$.

\begin{remark}{Notation}  If $S$ is any set, $\size(S)$ will denote
 the cardinality of $S$.
\end{remark}

\begin{lemma}\label{size} Let $(\M,\D)$ be a Cartan inclusion.  Then
$\size(R_a)=\size(\atom(\Z))$.
\end{lemma}
\begin{proof} Let $q:\atom(\D)\rightarrow \atom(\D)/R_a$ be the
  quotient map.  Define a map $\phi:\atom(\Z)\rightarrow \atom(\D)/R_a$
as follows.  For $A\in\atom(\Z)$, let $Q_1, Q_2\in\atom(\D)$ be the
unique atoms of $\D$ such that $A=\pil(Q_1)\pir(Q_2)$, see
Proposition~\ref{atomic}.  Now set $\phi(A)=q(Q_1)$. 

Observe that $\phi$ is onto: if $x\in\atom(\D)/R_a$ and $Q\in
q^{-1}(x)$, then $A:=\pil(Q)\pir(Q)\in\atom(Z)$ (see
Proposition~\ref{atomic}) and $\phi(A)=x$.

Fixing $x\in\atom(\D)/R_a$, Proposition~\ref{atomic} implies that
there is a bijection
\[\alpha_x: q^{-1}(x)\times q^{-1}(x)\rightarrow \phi^{-1}(x),\] where
$\alpha_x$ is the map given by $q^{-1}(x)\times q^{-1}(x)\ni (Q_1,
Q_2)\mapsto \pil(Q_1)\pir(Q_2)\in \phi^{-1}(x)$.  Since $\atom(\Z)$
and $R_a$ are the disjoint unions,
\[\atom(\Z)=\bigcup_{x\in \atom(\D)/R_a} \phi^{-1}(x)\dstext{and}
R_a=\bigcup_{x\in \atom(\D)/R_a} (q^{-1}(x)\times q^{-1}(x)),\] there
exists a bijection
\[\alpha: R_a\rightarrow \atom(\Z)\] given by
$\alpha(Q_1,Q_2)=\alpha_x(Q_1,Q_2)$ if $(Q_1,Q_2)\in q^{-1}(x)\times
q^{-1}(x).$

\end{proof}

For the following result we restrict our attention to the separably
acting case.

\begin{theorem}\label{samebbimod} For $i=1,2$, let $(\M_i,\D_i)$ be
Cartan pairs where $\M_i$ has separable predual, and let $\fS_i$ be
the lattice of all Bures-closed $\D_i$-bimodules contained in $\M_i$.
The following statements are equivalent.
\begin{enumerate}
\item There is a lattice isomorphism $\alpha$ of $\fS_1$ onto $\fS_2$.
\item There is a lattice isomorphism $\alpha'$ from the projection
lattice of $\Z_1$ onto the projection lattice of $\Z_2$.
\item There is a von Neumann algebra isomorphism $\Theta$ of $\Z_1$
onto $\Z_2$.
\item The atomic relations $R_{a,i}$ for $(\M_i,\D_i)$ satisfy
$\size(R_{a,1})=\size(R_{a,2}).$
\end{enumerate}
\end{theorem}
\begin{proof} The equivalence of (a) and (b) follows directly
  from~\cite[Theorem~2.5.8]{CameronPittsZarikianCaMASAvNAlNoNePrMeTh}.
  The equivalence of (b) and (c) is a piece of folklore about abelian
  von Neumann algebras.  (Here is a sketch of the non-trivial
  direction.  Suppose $\alpha'$ is an isomorphism of the projection
  lattices.  For every finite Boolean algebra, $\A\subseteq
  \text{proj}(\Z_1)$, $\alpha'|_\A$ extends uniquely to a \cstaralg\
  isomorphism $\Theta_\A$ of $C^*(\A)$ onto $C^*(\alpha'(\A))$.  As
  $\Z_1$ is the \cstar-inductive limit of the family $\{C^*(\A): \A
  \text{ a finite Boolean algebra of proj$(\Z_1)$}\}$ (with inclusion
  maps), the inductive limit $\Theta$ of the maps $\Theta_\A$ is an
  isomorphism of $\Z_1$ onto $\Z_2$.  But every isomorphism between
  von Neumann algebras is weak-$*$ continuous, so $\Z_1$ and $\Z_2$
  are isomorphic von Neumann algebras.)

If (c) holds, then $\atom(\Z_1)$ and $\atom(\Z_2)$ have the same
cardinality, so Lemma~\ref{size} implies (d) holds.  Conversely, if
(d) holds, then the atomic parts of $\Z_1$ and $\Z_2$ are isomorphic.
As $\Z_i$ are MASAs acting on separable Hilbert spaces,
Proposition~\ref{atomic}(c) implies the continuous parts of $\Z_1$ and
$\Z_2$ are isomorphic.  Therefore, $\Z_1$ is unitarily equivalent to
$\Z_2,$ and (c) holds.
\end{proof}

Theorem~\ref{samebbimod} is perhaps initially most remarkable when we
consider von Neumann algebras without atoms. For example, let
$(\M_1,\D_1)$ and $(\M_2,\D_2)$ be Cartan pairs where $\M_1$ is of
type II$_1$ and $\M_2$ is of type III$_\lambda$ for some
$\lambda$. Theorem~\ref{samebbimod} tells us that, while $\M_1$ and
$\M_2$ are quite different as von Neumann algebras, the lattice
of Bures-closed $\D_1$-bimodules of $\M_1$ is  isomorphic to the lattice of
Bures-closed $\D_2$-bimodules of $\M_2$. The following simple example
illustrates the situation regarding the atoms in
Theorem~\ref{samebbimod}.

\begin{remark}{Example} For any $n\in\bbN$, let $D_n\subseteq
M_n(\bbC)$ be the set of diagonal $n\times n$ matricies.  Let
$(\M_1,\D_1)=(M_2(\bbC), D_2)$, and let $(\M_2,\D_2)=(D_4,D_4)$.  Then
$\Z_1$ is isomorphic to $D_4$, so these Cartan pairs have isomorphic
lattices of bimodules.
\end{remark}

\begin{remark}{Remark} The non-separable case is complicated by the
fact that there are many isomorphism classes of non-atomic abelian von
Neumann algebras.  Indeed, if $\H$ is non-separable and $\D\subseteq
\bh$ is a non-atomic MASA with a unit cyclic vector $\xi$, then there
is a countable set $I$ such that $\D$ is isomorphic to the direct sum,
$\bigoplus_{i\in I} L^\infty(X_i,\mu_i)$, where $X_i=[0,1]^{A_i}$ is
a Cartesian product of the unit interval, $\mu_i$ is product
measure, and for at least one $i$, $A_i$ is a set with
$\size(A_i)>\aleph_0$ (see \cite{MaharamOnHoMeAl} and
\cite{SegalEqMeSp}).  A general MASA decomposes into a direct sum of
cyclic MASAs, hence there is a family $\{Q_\alpha\}_{\alpha\in \fI}$
of projections in $\D$, for which $Q_\alpha\D$ is isomorphic to
$L^\infty([0,1]^{A_\alpha})$.  Since $Q_\alpha$ is not minimal, the
arguments of Proposition~\ref{atomic} do not seem to apply, and it is
not clear how the statement of Theorem~\ref{samebbimod} should be
modified in the non-separable case.
\end{remark}

\def\cprime{$'$}


\def\cprime{$'$}
\providecommand{\bysame}{\leavevmode\hbox to3em{\hrulefill}\thinspace}
\providecommand{\MR}{\relax\ifhmode\unskip\space\fi MR }
\providecommand{\MRhref}[2]{%
  \href{http://www.ams.org/mathscinet-getitem?mr=#1}{#2}
}
\providecommand{\href}[2]{#2}

\end{document}